\documentclass[12pt,a4paper]{article}

\usepackage{graphicx}
\usepackage{amsmath,amsthm,amssymb}
\usepackage{esint}
\usepackage{mathrsfs}

\usepackage[T1]{fontenc}

\usepackage{hyperref}

\usepackage[left=2.2cm,right=2.2cm,top=3cm,bottom=3.5cm]{geometry}

\usepackage{color}
\usepackage{array}
\usepackage{multirow}

\usepackage{caption}
\usepackage{subcaption}

\theoremstyle{plain}
\newtheorem{thm}{Theorem}[section]
\newtheorem{prop}[thm]{Proposition}
\newtheorem{lemma}[thm]{Lemma}
\newtheorem{cor}[thm]{Corollary}

\theoremstyle{definition}
\newtheorem{defn}[thm]{Definition}

\theoremstyle{remark}
\newtheorem{rem}[thm]{Remark}

\numberwithin{equation}{section}
\usepackage{csquotes}
\usepackage[backend=bibtex8, style=numeric-comp, maxnames=5, sorting=nyt]{biblatex}
\usepackage{enumitem}

\newcommand{\R}{\mathbb{R}} 

\newcommand{\Grad}{\nabla}  
\newcommand{\Div}{{\rm div}\,} 

\newcommand{\dx}{\,{\rm d}\xx}
\newcommand{\dxcomp}{\,{\rm d}x}
\newcommand{\dycomp}{\,{\rm d}y}
\newcommand{\dt}{\,{\rm d}t}  
\newcommand{\dr}{\,{\rm d}r} 
\newcommand{\ds}{\,{\rm d}s} 

\renewcommand{\rho}{\varrho}
\newcommand{\half}{\tfrac{1}{2}}  
\newcommand{\Cc}{C^\infty_{\rm c}}

\newcommand{\wild}{{\rm wild}}
\newcommand{\example}{{\rm ex}} % subscript for the counterexample

\newcommand{\name}[1]{\textsc{#1}} 
\newcommand{\uu}{\mathbf{u}} 
\newcommand{\xx}{\mathbf{x}} 
\newcommand{\vphi}{\boldsymbol{\varphi}}
\begin{document}
	
	\title{Failure of the least action admissibility principle in the context of the compressible Euler equations} 
	
	\author{Simon Markfelder\footnote{e-mail: \textsf{simon.markfelder@uni-konstanz.de}} \and Valentin Pellhammer\footnote{Corresponding author, e-mail: \textsf{valentin.pellhammer@uni-konstanz.de}}} 
	
	\date{\today}
	
	\maketitle

    \bigskip

    \centerline{Universit\"at Konstanz, Department of Mathematics and Statistics,} 

    \centerline{Post office box: 199, 78457 Konstanz, Germany} 

    \bigskip
	
	\begin{abstract} 
        Finding a proper solution concept for the multi-dimensional barotropic compressible Euler equations and related systems is still an unsolved problem. As revealed by convex integration, the classical notion of an admissible weak solutions (also known as weak entropy solutions) does not lead to uniqueness and allows for solutions which do not seem to be physical. For this reason, people have studied additional criteria in view of their ability to rule out the counterintuitive solutions generated by convex integration. Recently, in [H.~Gimperlein, M.~Grinfeld, R.~J.~Knops and M.~Slemrod: The least action admissibility principle, arXiv: 2409.07191 (2024)] it was suggested that the least action admissibility principle serves as the desired selection criterion. In this paper, however, we show that the least action admissibility principle rules out the solution which is intuitively the physically relevant one. Consequently, one either has to reconsider one's intuition, or the least action admissibility principle must be discarded. 
	\end{abstract}
	
	\bigskip
	
	
	\bigskip
	

    % 76N10  : Existence, uniqueness, and regularity theory for compressible fluids and gas dynamics
    % 35Q31  : Euler equations 
    % 35D30  : Weak solutions to PDEs
    % 35L65  : Hyperbolic conservation laws
    
	\bigskip
	
	\tableofcontents
	
	\section{Introduction and main result} \label{sec:intro} 
    
    The quest for a good solution concept for multi-dimensional systems of conservation laws is an outstanding open problem. While for the scalar case as well as for one-dimensional systems, the notion of an \emph{admissible weak solution} (also known as \emph{weak entropy solution}) has led to a satisfactory well-posedness theory (see \name{Kru{\v z}kov}~\cite{Kruzkov70}, \name{Glimm}~\cite{Glimm65}, \name{Bressan}~et~al.~\cite{BreCraPic00}), the method of convex integration has shown that this notion does not lead to uniqueness in the context of multi-dimensional systems (see \name{De~Lellis}-\name{Sz{\'e}kelyhidi}~\cite{DelSze10}). 

    In particular, the barotropic compressible Euler equations in two space dimensions 	
	\begin{align}\label{eq:euler}
		\begin{split}
			\partial_t\rho + \Div (\rho \uu) &= 0,\\
			\partial_t(\rho\uu) + \Div(\rho \uu\otimes\uu)+\nabla p(\rho)&=0
		\end{split}
	\end{align}
    have been studied regarding non-uniqueness via convex integration. System \eqref{eq:euler} describes a fluid with mass density $\rho=\rho(t,\xx)\in\R^+$ and velocity field $\uu=\uu(t,\xx)=(u(t,\xx),v(t,\xx))\in\R^2$. Here $t\in [0,T)$ (with maximal time $T>0$) and $\xx=(x,y)\in \R^2$ denote time and space, respectively. The pressure $p$ is related to the density $\rho$ by the equation of state. For our purposes we choose
	\begin{align} \label{eq:pressure}
		p(\rho) = \rho^2,
	\end{align}
	which is a special case of the polytropic pressure law $p(\rho) = K\rho^\gamma$, $K>0$, $\gamma>1$ with $K=1$ and $\gamma=2$. One is usually interested in the Cauchy problem, i.e.~one endows the Euler equations \eqref{eq:euler} with initial condition 
    \begin{align}
		\begin{split}\label{eq:initialData}
			\rho(0,\xx) &= \rho_{0}(\xx),\\
			\uu(0,\xx) &= \uu_{0}(\xx),
		\end{split}
	\end{align}		
	where the functions $\rho_{0}\in L^{\infty}(\R^2;\R^+)$ and $\uu_{0}\in L^{\infty}(\R^2;\R^2)$ are given. 

    In order to define admissible weak solutions to the initial value problem \eqref{eq:euler}, \eqref{eq:initialData}, one has to specify the entropies for system \eqref{eq:euler}. It turns out that there is only one relevant entropy, namely the energy $\half \rho|\uu|^2 + P(\rho)$ with corresponding energy flux $\big(\half \rho |\uu|^2 + P(\rho) + p(\rho) \big) \uu$. Here, the pressure potential $P$ is defined by
    \[
		P(\rho) = \rho \int^{\rho}\frac{p(r)}{r^2}\dr,
	\]
    which turns into $P(\rho) = \rho^2$ due to our choice \eqref{eq:pressure} of the pressure law. Now, admissible weak solutions are defined as follows.
    \begin{defn}
        A pair $(\rho,\uu)\in L^\infty ((0,T)\times\R^2;\R^+\times\R^2)$ is a \emph{weak solution} of the initial value problem \eqref{eq:euler}, \eqref{eq:initialData} if the followng equations are satisfied for all test functions $(\phi,\vphi)\in \Cc([0,T)\times \R^2;\R\times \R^2)$:
        \begin{align*}
            \int_0^T \int_{\R^2}\Big[ \rho \partial_t \phi +\rho \uu \cdot\nabla \phi \Big] \dx \dt + \int_{\R^2}\rho_0\phi(0,\cdot) \dx &=0;\\
            \int_0^T \int_{\R^2}\Big[ \rho\uu\cdot\partial_t \vphi +\rho\uu\otimes\uu:\nabla \vphi +p(\rho)\Div \vphi \Big] \dx \dt + \int_{\R^2} \rho_0\uu_0\cdot \vphi(0,\cdot) \dx &=0.
        \end{align*}
        A weak solution $(\rho,\uu)$ is called \emph{admissible} if 
        \begin{align*}
            \int_0^T \int_{\R^2}\left[ \Big(\half\rho|\uu|^2 + P(\rho)\Big)\partial_t\psi + \Big(\half\rho|\uu|^2 +P(\rho) + p(\rho)\Big)\uu\cdot\nabla\psi \right]\dx\dt & \\
            +\int_{\R^2}\Big(\half\rho_0|\uu_0|^2+P(\rho_0)\Big)\psi(0,\cdot)\dx &\geq 0 
        \end{align*}
        for all $\psi\in\Cc([0,T)\times \R^2;\R_0^+)$.
    \end{defn}

    In this paper, we focus on the Riemann problem for \eqref{eq:euler}, i.e.~we choose the initial data \eqref{eq:initialData} to be given as
	\begin{align}\label{eq:RiemannData}
		(\rho_{0},\uu_{0})(\xx) = \begin{cases}
			(\rho_-,\uu_-),& y<0,\\
			(\rho_+,\uu_+),& y>0,
		\end{cases}
	\end{align}	
	with constant states $\rho_\pm\in\R^+$ and $\uu_\pm\in\R^2$. The data \eqref{eq:RiemannData} does not depend on the first space variable $x$. Hence, one may want to look for solutions which are independent of $x$ for any positive time $t>0$. This ansatz boils the problem down to a one-dimensional Riemann problem which can be solved by well-known methods and leads to a self-similar solution composed of constant states, shock waves, contact discontinuities and rarefaction waves. It is not difficult to see that this self-similar solution (interpreted as a function of $(t,\xx)\in[0,T)\times\R^2$ which in fact only depends on $(t,y)$) serves as an admissible weak solution for the two-dimensional problem, too. We will denote this solution as the \emph{1-D solution} to the Riemann problem \eqref{eq:euler}, \eqref{eq:RiemannData}.

    Utilizing the method of convex integration as established by \name{De~Lellis}-\name{Sz{\'e}kelyhidi}~\cite{DelSze09,DelSze10}, \name{Chiodaroli}-\name{De~Lellis}-\name{Kreml}~\cite{ChiDelKre15} showed that for some initial states $(\rho_\pm,\uu_\pm)$ there exist infinitely many admissible weak solutions to \eqref{eq:euler}, \eqref{eq:RiemannData} in addition to the 1-D solution. These solutions are genuinely two-dimensional, i.e.~they do depend on both $x$ and $y$ (in addition to $t$). In other words, they exhibit a break of symmetry which is counterintuitive. For this reason and due to the non-uniqueness, one has checked whether additional criteria are able to rule out the seemingly unphysical solutions, and maybe restore uniqueness. The natural candidate that one might expect to be selected by such additional criteria is the 1-D solution. 

    In particular, \name{Chiodaroli}-\name{Kreml}~\cite{ChiKre14} consider the \emph{global maximal dissipation criterion} (also known as the \emph{entropy rate admissibility criterion}) which was proposed by \name{Dafermos}~\cite{Dafermos73}. In \cite{ChiKre14} solutions are constructed via convex integration which dissipate more energy than the 1-D solution. Consequently, the global maximal dissipation criterion does not favor the 1-D solution. 
    
    Moreover, a local version of the maximal dissipation criterion has been studied by the first author~\cite{Markfelder24}. The result is the same as in the global case: there exists a solution generated by convex integration which beats the 1-D solution, i.e.~the local version of the maximal dissipation criterion does not select the 1-D solution too. 

    Motivated by the classical least action principle used in quantum mechanics,  \name{Gimperlein}-\name{Grinfeld}-\name{Knops}-\name{Slemrod}~\cite{GGKS24pre} recently proposed another criterion to select the physically relevant solution to the Riemann problem \eqref{eq:euler}, \eqref{eq:RiemannData}. For any admissible weak solution $(\rho,\uu)\in L^\infty ((0,T)\times\R^2;\R^+\times\R^2)$, they consider its \emph{action} given by
	\begin{align}
		\mathcal{A}[\rho,\uu] := \int_{0}^{T}\int_{-L_1}^{L_1}\int_{-L_2}^{L_2} \Big(\half \rho|\uu|^2 - P(\rho) \Big)\dycomp\dxcomp\dt, \label{eq:actionDef}
	\end{align}
	  where $T>0$ is given and $L_1,L_2>0$ are chosen suitably.

    \begin{rem} \label{rem:choice-of-L}
        Goal of this paper is to prove Theorem~\ref{thm:mainTheorem} below. To this end, we only want to compare the action of two particular solutions. We may thus choose $L_2>0$ sufficiently large such that the two solutions coincide for all $(t,\xx)\in [0,T)\times \R^2$ with $|y|>L_2$. Moreover, the two solutions will be periodic in $x$ with, say, period $2$. Consequently, we will choose $L_1=1$. Alternatively, one may periodize our solutions also in $y$ by proceeding as in \cite[Sect.~6]{ChiKre14}. 
    \end{rem}
      
    We state the criterion that was proposed in \cite{GGKS24pre} in the following definition.
    
	\begin{defn} \label{defn:leastActionPrinciple}
		An admissible weak solution $(\rho,\uu)$ to \eqref{eq:euler}, \eqref{eq:RiemannData} fulfills the \emph{least action admissibility principle} if there is no other admissible weak solution $(\widetilde{\rho},\widetilde{\uu})$ emerging from the same initial data \eqref{eq:RiemannData} with the property 
		\begin{align}\label{eq:ActionInequality}
			\mathcal{A}[\widetilde{\rho},\widetilde{\uu}]<\mathcal{A}[\rho,\uu].
		\end{align}
	\end{defn}

    The results in \cite{GGKS24pre} suggest that the 1-D solution satisfies the least action admissibility principle. The aim of this note is to show that this is wrong. More precisely, we prove the following.
    
	\begin{thm}\label{thm:mainTheorem}
		For each $T>0$, there exist\footnote{In fact, such Riemann data $(\rho_\pm,\uu_\pm)$ are independent of $T>0$.} Riemann initial data $(\rho_\pm,\uu_\pm)\in\R^+\times\R^2$ such that the 1-D solution to \eqref{eq:euler}, \eqref{eq:RiemannData} does not fulfill the least action admissibility principle.
	\end{thm}

    Similar to the strategy used in \cite{ChiKre14} and \cite{Markfelder24}
    to reject criteria regarding maximal dissipation, we will provide explicit Riemann initial data $(\rho_\pm,\uu_\pm)\in\R^+\times\R^2$ to which we construct a solution $(\widetilde{\rho},\widetilde{\uu})$ via convex integration whose action is lower than the action of the 1-D solution $(\rho,\uu)$, i.e.~we show that \eqref{eq:ActionInequality} holds. This proves Theorem~\ref{thm:mainTheorem}.

    The consequence of Theorem~\ref{thm:mainTheorem} is the same as for the criteria regarding maximal dissipation studied in \cite{ChiKre14} and \cite{Markfelder24}: if one insists that the physically relevant solution to the Riemann problem \eqref{eq:euler}, \eqref{eq:RiemannData} satisfies the least action admissibility principle, then the 1-D solution is ruled out. Or, if one expects the 1-D solution to be the physically relevant solution, then the least action admissibility principle must be discarded. 

    This paper is organized as follows. 
    Using convex integration, we construct a class of admissible weak solutions with a very specific shape in Section \ref{sec:wildSolutions}. In Section~\ref{sec:proof} we choose explicit initial data \eqref{eq:RiemannData} and, using the results of Section~\ref{sec:wildSolutions}, obtain additional weak solutions alongside the 1-D solution. Furthermore, we show that these solutions beat the 1-D solution in the sense that they have a smaller action, thereby proving Theorem \ref{thm:mainTheorem}. Finally, in Section \ref{sec:concluding-remarks}, we illustrate the mechanism that enables the success of the proof.

	\section{Shape of the solutions constructed in this paper}\label{sec:wildSolutions} 

    The approach of applying convex integration in the context of the Riemann problem \eqref{eq:euler}, \eqref{eq:RiemannData}, which was established in \cite{ChiDelKre15}, leads to the following statement. 
	
	\begin{prop}\label{prop:convexIntegration} 
		Let $\rho_\pm\in\R^+$, $\uu_\pm\in\R^2$ and $T_0>0$. Assume that there exist numbers $\mu_0,\mu_1\in\R$, $\rho_1\in\R^+$, $\uu_1\in \R^2$, $\gamma_1,\delta_1\in\R$ and $C_1\in\R^+$ which fulfill the following algebraic equations and inequalities:
		\begin{itemize} 
			\item Order of speeds:
			\[
			\mu_0<\mu_1;
			\]
            
			\item Rankine Hugoniot conditions on the left interface:
			\begin{align*}
				\mu_0(\rho_- - \rho_1) &= \rho_-v_- - \rho_1v_1 ; \\
				\mu_0(\rho_-u_- - \rho_1u_1) &= \rho_-u_-v_- - \rho_1\delta_1 ; \\
				\mu_0(\rho_-v_--\rho_1v_1) &=\rho_-(v_-)^2 - \rho_1\left(\frac{C_1}{2}-\gamma_1\right)+p( \rho_-) - p(\rho_1);
			\end{align*}
            
			\item Rankine Hugoniot conditions on the right interface:
			\begin{align*}
				\mu_1(\rho_1 - \rho_+) &= \rho_1v_1 - \rho_+v_+ ;\\
				\mu_1(\rho_1u_1 - \rho_+u_+) &= \rho_1\delta_1 - \rho_+u_+v_+ ;\\
				\mu_1(\rho_1v_1-\rho_+v_+) &=\rho_1\left(\frac{C_1}{2}-\gamma_1\right) - \rho_+(v_+)^2 + p(\rho_1) - p(\rho_+);
			\end{align*}
            
			\item Subsolution condition:
			\begin{align}
				C_1 - (u_1)^2 - (v_1)^2&>0; \label{eq:subsolution-condition} \\
				\left(\frac{C_1}{2} - (u_1)^2 + \gamma_1\right)\left(\frac{C_1}{2} - (v_1)^2 - \gamma_1\right) - (\delta_1 - u_1v_1)^2&>0 \notag ;
			\end{align}
            
			\item Admissibility condition on the left interface:
			\begin{align*}
				&\mu_0\left(\rho_-\frac{|\uu_-|^2}{2} + P(\rho_-) - \rho_1\frac{C_1}{2}-P(\rho_1)\right)\\
				&\leq \left(\rho_-\frac{|\uu_-|^2}{2} + P(\rho_-) + p(\rho_-)\right)v_- - \left(\rho_1\frac{C_1}{2} + P(\rho_1) + p(\rho_1)\right)v_1;
			\end{align*}
            
			\item Admissibility condition on the right interface:
			\begin{align*}
				&\mu_1\left(\rho_1 \frac{C_1}{2} + P(\rho_1) - \rho_+\frac{|\uu_+|^2}{2} - P(\rho_+)\right)\\
				&\leq \left(\rho_1\frac{C_1}{2}+P(\rho_1)+p(\rho_1)\right)v_1 - \left(\rho_+\frac{|\uu_+|^2}{2} + P(\rho_+) + p(\rho_+)\right)v_+.
			\end{align*}
		\end{itemize}		
		Then there exist infinitely many admissible weak solutions $(\rho,\uu)\in L^\infty((0,T_0)\times \R^2;\R^+\times\R^2)$ of the barotropic Euler equations \eqref{eq:euler} with initial data \eqref{eq:initialData} in the following sense:		
		\begin{align}
            \int_0^{T_0}\int_{\R^2} \Big[ \rho \partial_t \phi + \rho \uu \cdot \Grad\phi \Big] \dx\dt + \int_{\R^2} \rho_0 \phi(0,\cdot) \dx - \int_{-\infty}^{\mu_0 T_0} \int_{\R} \rho_- \phi(T_0,\cdot) \dxcomp \dycomp & \notag \\
            - \int_{\mu_0 T_0}^{\mu_1 T_0} \int_{\R} \rho_1 \phi(T_0,\cdot) \dxcomp \dycomp - \int_{\mu_1 T_0}^{\infty} \int_{\R} \rho_+ \phi(T_0,\cdot) \dxcomp \dycomp &= 0, \label{eq:weakSolinThm-mass} \\
            \int_0^{T_0}\int_{\R^2} \Big[ \rho \uu \cdot \partial_t \vphi + \rho \uu \otimes \uu : \Grad\vphi + p(\rho) \Div \vphi \Big] \dx\dt + \int_{\R^2} \rho_0 \uu_0 \cdot \vphi(0,\cdot) \dx & \notag \\
            - \int_{-\infty}^{\mu_0 T_0} \int_{\R} \rho_- \uu_- \cdot \vphi(T_0,\cdot) \dxcomp \dycomp - \int_{\mu_0 T_0}^{\mu_1 T_0} \int_{\R} \rho_1 \uu_1 \cdot \vphi(T_0,\cdot) \dxcomp \dycomp & \notag \\ 
            - \int_{\mu_1 T_0}^{\infty} \int_{\R} \rho_+ \uu_+ \cdot \vphi(T_0,\cdot) \dxcomp \dycomp &= 0 , \label{eq:weakSolinThm-mom} \\
			\int_0^{T_0}\int_{\R^2} \Big[ \Big(\half \rho |\uu|^2 + P(\rho) \Big) \partial_t \psi + \Big(\half \rho |\uu|^2 + P(\rho) + p(\rho) \Big) \uu \cdot \Grad\psi \Big] \dx\dt & \notag \\
            + \int_{\R^2} \Big(\half \rho_0 |\uu_0|^2 + P(\rho_0) \Big) \psi(0,\cdot) \dx - \int_{-\infty}^{\mu_0 T_0} \int_{\R} \Big(\half \rho_- |\uu_-|^2 + P(\rho_-) \Big) \psi(T_0,\cdot) \dxcomp \dycomp & \notag \\ 
            - \int_{\mu_0 T_0}^{\mu_1 T_0} \int_{\R} \Big(\half \rho_1 C_1 + P(\rho_1) \Big)\psi(T_0,\cdot) \dxcomp \dycomp - \int_{\mu_1 T_0}^{\infty} \int_{\R} \Big(\half \rho_+ |\uu_+|^2 + P(\rho_+) \Big) \psi(T_0,\cdot) \dxcomp \dycomp &\geq 0 , \label{eq:weakSolinThm-en} 
		\end{align}		
        for all test functions $(\phi,\vphi,\psi)\in \Cc([0,T_0]\times \R^2;\R\times \R^2\times \R^+_0)$. Additionally all these solutions fulfill
		\begin{itemize}
			\item $\rho(t,\xx) = \rho_1$ and $|\uu(t,\xx)|^2 = C_1$ for a.e.~$(t,\xx)\in\Gamma_1$,
            \item $(\rho,\uu)(t,\xx)=(\rho_\pm,\uu_\pm)$ for a.e.~$(t,\xx)\in\Gamma_\pm$,
		\end{itemize}
        where $\Gamma_-,\Gamma_1,\Gamma_+$ are given by
        \begin{align*}
            \Gamma_- &:= \{(t,\xx): y<\mu_0 t,\ 0<t<T_0 \}, \\
		      \Gamma_1 &:= \{(t,\xx): \mu_0t<y<\mu_1 t,\ 0<t<T_0 \}, \\
            \Gamma_+ &:= \{(t,\xx): y>\mu_1 t,\ 0<t<T_0 \}. 
        \end{align*}
        For any $L>0$, infinitely many among these solutions are periodic in $x$ with period $L$. 
	\end{prop} 

    Proposition~\ref{prop:convexIntegration} can be proved by applying convex integration in the region $\Gamma_1$. We refer to \cite[Props.~3.6 and 5.1]{ChiDelKre15}, \cite[Props.~3.1 and 4.1]{ChiKre14} or \cite[Thm.~7.3.4 and Prop.~7.3.5]{Markfelder} for more details. 

    \begin{rem}
        The reader may have noticed that the solutions studied in the aforementioned papers are defined globally in time, in contrast to Proposition~\ref{prop:convexIntegration} which considers solutions on a finite time interval $[0,T_0]$. It is however easy to observe that the method of \cite{ChiDelKre15} can be also applied on a finite time interval $[0,T_0]$. 
    \end{rem}

	Theorem \eqref{prop:convexIntegration} yields a solution on the time interval $[0,T_0]$. To define a global-in-time solution for the Riemann problem \eqref{eq:euler}, \eqref{eq:RiemannData}, we consider another initial value problem at the time level $t=T_0$ and will glue the two solutions together. 
	
	\begin{lemma}\label{lemma:Solution_t>T} 
		Let $\rho_\pm\in\R^+$, $\uu_\pm\in\R^2$, $T>T_0>0$, $\mu_0<\mu_1$, $\rho_1\in\R^+$ and $\uu_1\in\R^2$. Then the initial value problem \eqref{eq:euler} together with data
		\begin{align}\label{eq:initialDataAtT0}
			(\rho,\uu)(T_0,\xx) :=\begin{cases}
				(\rho_-,\uu_-),&y<\mu_0 T_0, \\
				(\rho_1,\uu_1),&\mu_0 T_0<y<\mu_1 T_0, \\
				(\rho_+,\uu_+),&\mu_1 T_0<y,
			\end{cases}
		\end{align}
		has an admissible weak solution $(\rho,\uu) \in L^\infty((T_0,T)\times\R^2;\R^+\times \R^2)$. In particular, the following equations and inequalities hold for all test functions $(\phi,\vphi,\psi)\in \Cc([T_0,T)\times \R^2;\R\times \R^2\times \R^+_0)$: 
        \begin{align}
            \int_{T_0}^T \int_{\R^2} \Big[ \rho \partial_t \phi + \rho \uu \cdot \Grad\phi \Big] \dx\dt + \int_{-\infty}^{\mu_0 T_0} \int_{\R} \rho_- \phi(T_0,\cdot) \dxcomp \dycomp + \int_{\mu_0 T_0}^{\mu_1 T_0} \int_{\R} \rho_1 \phi(T_0,\cdot) \dxcomp \dycomp & \notag \\
            + \int_{\mu_1 T_0}^{\infty} \int_{\R} \rho_+ \phi(T_0,\cdot) \dxcomp \dycomp &= 0; \label{eq:weakSolinLemma-mass} \\
            \int_{T_0}^T \int_{\R^2} \Big[ \rho \uu \cdot \partial_t \vphi + \rho \uu \otimes \uu : \Grad\vphi + p(\rho) \Div \vphi \Big] \dx\dt + \int_{-\infty}^{\mu_0 T_0} \int_{\R} \rho_- \uu_- \cdot \vphi(T_0,\cdot) \dxcomp \dycomp & \notag \\
            + \int_{\mu_0 T_0}^{\mu_1 T_0} \int_{\R} \rho_1 \uu_1 \cdot \vphi(T_0,\cdot) \dxcomp \dycomp + \int_{\mu_1 T_0}^{\infty} \int_{\R} \rho_+ \uu_+ \cdot \vphi(T_0,\cdot) \dxcomp \dycomp &= 0; \label{eq:weakSolinLemma-mom} \\
			\int_{T_0}^T \int_{\R^2} \Big[ \Big(\half \rho |\uu|^2 + P(\rho) \Big) \partial_t \psi + \Big(\half \rho |\uu|^2 + P(\rho) + p(\rho) \Big) \uu \cdot \Grad\psi \Big] \dx\dt & \notag \\
            + \int_{-\infty}^{\mu_0 T_0} \int_{\R} \Big(\half \rho_- |\uu_-|^2 + P(\rho_-) \Big) \psi(T_0,\cdot) \dxcomp \dycomp & \notag \\ 
            + \int_{\mu_0 T_0}^{\mu_1 T_0} \int_{\R} \Big(\half \rho_1 |\uu_1|^2 + P(\rho_1) \Big)\psi(T_0,\cdot) \dxcomp \dycomp & \notag \\
            + \int_{\mu_1 T_0}^{\infty} \int_{\R} \Big(\half \rho_+ |\uu_+|^2 + P(\rho_+) \Big) \psi(T_0,\cdot) \dxcomp \dycomp &\geq 0. \label{eq:weakSolinLemma-en} 
		\end{align}

		\begin{proof}
			The given data does not depend on $x$. We can hence use the classical result of \name{Glimm}~\cite{Glimm65} for initial value problems in one space dimension to obtain an admissible weak solution to the corresponding one-dimensional Cauchy problem. This solution also solves the two-dimensional problem \eqref{eq:euler}, \eqref{eq:initialDataAtT0} (by interpreting the one-dimensional solution as a function of $(t,\xx)\in [T_0,T)\times \R^2$ which is independent of $x$).
		\end{proof}
	\end{lemma}

    Now we are able to glue a solution from Proposition~\ref{prop:convexIntegration}, which lives on the time interval $[0,T_0]$, to a solution from Lemma~\ref{lemma:Solution_t>T} which exists on $[T_0,T)$.
	
	\begin{cor}\label{cor:gluing} 
		Let $\rho_\pm\in\R^+$, $\uu_\pm\in\R^2$ and $T>T_0>0$. Assume that there exist numbers $\mu_0,\mu_1\in\R$, $\rho_1\in\R^+$, $\uu_1\in \R^2$, $\gamma_1,\delta_1\in\R$ and $C_1\in\R^+$ that satisfy the algebraic equations and inequalities given in Proposition~\ref{prop:convexIntegration}. Let $(\rho_{t<T_0},\uu_{t<T_0})$ be a solution resulting from Proposition~\ref{prop:convexIntegration} and $(\rho_{t>T_0},\uu_{t>T_0})$ be the solution given by Lemma~\ref{lemma:Solution_t>T}. Then $(\rho,\uu) \in L^\infty((0,T)\times\R^2;\R^+\times \R^2)$, given by
		\begin{align}\label{eq:counterExampleSolution}
			(\rho,\uu)(t,\xx) : = \begin{cases}
				(\rho_{t<T_0},\uu_{t<T_0})(t,\xx),&0<t<T_0, \\
				(\rho_{t>T_0},\uu_{t>T_0})(t,\xx),&T_0<t<T,
			\end{cases}
		\end{align}
		is an admissible weak solution to the Riemann problem \eqref{eq:euler}, \eqref{eq:RiemannData}.
		
		\begin{proof}
			Adding \eqref{eq:weakSolinThm-mass} and \eqref{eq:weakSolinThm-mom} to \eqref{eq:weakSolinLemma-mass} and \eqref{eq:weakSolinLemma-mom}, respectively, shows that $(\rho,\uu)$ is indeed a weak solution. Adding \eqref{eq:weakSolinThm-en} to \eqref{eq:weakSolinLemma-en}, we find
            \begin{align*}
                &\int_{0}^T \int_{\R^2} \Big[ \Big(\half \rho |\uu|^2 + P(\rho) \Big) \partial_t \psi + \Big(\half \rho |\uu|^2 + P(\rho) + p(\rho) \Big) \uu \cdot \Grad\psi \Big] \dx\dt \\
                & + \int_{\R^2} \Big(\half \rho_0 |\uu_0|^2 + P(\rho_0) \Big) \psi(0,\cdot) \dx \\ 
                &\geq \int_{\mu_0 T_0}^{\mu_1 T_0} \int_{\R} \half \rho_1 \Big(C_1-|\uu_1|^2 \Big)\psi(T_0,\cdot) \dxcomp \dycomp \geq 0
            \end{align*}
            for all $\psi\in \Cc([0,T)\times \R^2;\R^+_0)$, where the latter inequality is due to \eqref{eq:subsolution-condition}. Thus, the weak solution $(\rho,\uu)$ is even admissible.
		\end{proof}
	\end{cor}
	
	\begin{rem} \label{rem:glueingSolutions}
        The basic idea in Corollary~\ref{cor:gluing} is to glue a convex integration solution on the time interval $[0,T_0]$ to a classical solution on $[T_0,T)$. Along the surface 
		\[
		      \{(t,\xx): t=T_0,\ \mu_0T_0<y<\mu_1T_0\},
		\]
		the kinetic energy has a jump discontinuity and jumps from $\half\rho_1C_1$ to $\half\rho_1|\uu_1|^2$. According to \eqref{eq:subsolution-condition}, this is a jump from a higher to a lower value, which ensures that the solution still fulfills the energy inequality. In a slightly different context, namely when convex integration is applied on a space-time cylinder $[0,T)\times \Omega$, a similar energy jump is observed at initial time $t=0$. The latter jump is however a jump from a lower to a higher value which destroys validity of the energy inequality, see e.g.~\name{Feireisl}~\cite[Sect.~13.2.3]{Feireisl16_1}. When designing our solution, we were inspired by such an unphysical energy jump, which becomes physical when time is reversed. 
	\end{rem}

	\section{Proof of Theorem \ref{thm:mainTheorem}} \label{sec:proof}

    In order to prove Theorem~\ref{thm:mainTheorem}, we will present a particular choice of Riemann initial states $(\rho_\pm,\uu_\pm)\in\R^+\times\R^2$ and construct a solution $(\rho_{\example},\uu_{\example})$ of the form \eqref{eq:counterExampleSolution} (see Corollary~\ref{cor:gluing}) whose action $\mathcal{A}[\rho_{\example},\uu_{\example}]$ is smaller than the action $\mathcal{A}[\rho_{1d},\uu_{1d}]$ of the corresponding 1-D solution $(\rho_{1d},\uu_{1d})$. 
	
	\subsection{Special Riemann data and the corresponding 1-D solution}\label{subsec:RiemannProblemExplicit}
	
	Consider Riemann initial data \eqref{eq:RiemannData} with
	\begin{align}\label{eq:explicitRiemannData}
        \begin{split}
        	\rho_- &= 1,\quad u_-=0,\quad  v_-=\frac{57\sqrt{35}}{10} + \frac{59 \sqrt{915}}{30}, \\
		      \rho_+ &= \rho_-,\quad u_+=0,\quad v_+=-v_-.
        \end{split}
	\end{align}
	The corresponding 1-D solution consists of two shock waves and is given by	
	\begin{align*} 
		(\rho_{1d},\uu_{1d})(t,\xx) = \begin{cases}
			(\rho_-,\uu_-),& y<-\sigma t,\\ 
			(\rho_M,0),& -\sigma t<y<\sigma t,\\
			(\rho_+,\uu_+),& \sigma t<y,
		\end{cases} 
	\end{align*}	
    where $\rho_M$ is determined by 
    \begin{align*}
        \rho_M &> \max\{ \rho_\pm \}, \\
        0 &= v_- - \sqrt{\frac{(\rho_M-\rho_-)(p(\rho_M) - p(\rho_-))}{\rho_M\rho_-}},
    \end{align*}
    and $\sigma= \frac{\rho_- v_-}{\rho_M - \rho_-}>0$, see e.g.~\cite[Lemma~2.4]{ChiKre14}, \cite[Prop.~7.1.1]{Markfelder} or standard textbooks on hyperbolic conservation laws like the one by \name{Dafermos}~\cite{Dafermos}. It is simple to verify that
    \begin{align}\label{eq:estimateRhoMSigma}
        \rho_M<94\quad \text{and}\quad \sigma<\frac{11}{10}.  %es ist rhoM ~ 93.717937 und \sigma ~1.0053
    \end{align}

    We will compute the action $\mathcal{A}[\rho_{1d},\uu_{1d}]$ of the 1-D solution $(\rho_{1d},\uu_{1d})$ in Section~\ref{subsec:ActionComp} below.

	\subsection{Construction of another solution} \label{subsec:constructionCounterEx}

    Next we construct another solution to the Riemann problem \eqref{eq:euler}, \eqref{eq:RiemannData} with Riemann data \eqref{eq:explicitRiemannData} by invoking Corollary~\ref{cor:gluing}. Let us first fix $T>0$ and $T_0=\frac{T}{2}$. We have to find $\mu_0,\mu_1\in\R$, $\rho_1\in\R^+$, $\uu_1\in \R^2$, $\gamma_1,\delta_1\in\R$ and $C_1\in\R^+$ that satisfy the algebraic equations and inequalities given in Proposition~\ref{prop:convexIntegration}. We choose the following values
    \begin{align*}
		\rho_1 &= 3, & \gamma_1&= -\frac{1121\sqrt{1281}}{40} - \frac{28013}{24},\\
		u_1 &= 0, & \delta_1&=0,\\
		v_1 &= 0, & C_1 &= \frac{1121\sqrt{1281}}{20} + \frac{28037}{12}. 
	\end{align*}
	%\begin{align*}
	%	\tilde{\eps}_1 = 1,\quad
	%	\eps_1 =  \frac{1121\sqrt{1281}}{20} + \frac{28025}{12}.
	%\end{align*}
	Furthermore, we choose
	\[
	\mu_0  = -\frac{57\sqrt{35}}{20} - \frac{59\sqrt{915}}{60} ,\qquad \mu_1  = -\mu_0.
	\]	
	One can easily verify that the equations and inequalities given in Proposition~\ref{prop:convexIntegration} are satisfied. Hence, Proposition~\ref{prop:convexIntegration} and Corollary~\ref{cor:gluing} give us infinitely many solutions to the Riemann problem \eqref{eq:euler}, \eqref{eq:RiemannData} with data \eqref{eq:explicitRiemannData} of the form \eqref{eq:counterExampleSolution}.
	
	We choose one of these solutions and denote it by $(\rho_{\example},\uu_{\example})$, and will further investigate its structure for $T_0<t<T$. Therefore, we consider the solution to the initial value problem at $t=T_0$ with initial data \eqref{eq:initialDataAtT0} containing two jump discontinuities located at
	\begin{align}\label{eq:pointsy_0y_1}
	   y_0 = \mu_0T_0\quad\text{and}\quad y_1 =  \mu_1T_0 = -y_0,
	\end{align}
	respectively.
	Each of these discontinuities can be regarded as a separate Riemann problem. One can easily verify\footnote{For the classical solution of one-dimensional Riemann problems to the Euler equations \eqref{eq:euler}, we again refer to \cite[Lemma~2.4]{ChiKre14}, \cite[Prop.~7.1.1]{Markfelder} or standard textbooks on hyperbolic conservation laws like \cite{Dafermos}.} that the solution for each of these Riemann problems consists of two shock waves as is indicated in Figure \ref{fig:solutionInSpaceTime}. The corresponding intermediate states are given as follows
	\begin{align*} 
		\rho_2 &= 60, & u_2&=0, & v_2 &= \frac{57\sqrt{35}}{10},\\
		\rho_3&=\rho_2, & u_3&=0, & v_3 &= -v_2,
	\end{align*}
	with corresponding speeds
	\begin{align*}
        \mu_2 = -\frac{\sqrt{915}}{30} + \frac{57\sqrt{35}}{10},\qquad
		\mu_3 = 6\sqrt{35},\qquad \mu_4 =-\mu_3,\qquad \mu_5 = -\mu_2.
	\end{align*} 
    See Figure \ref{fig:solutionInSpaceTime} for notation.
	The two discontinuities emerging from $y_0$ and $y_1$ (cf.~\eqref{eq:pointsy_0y_1}), with speeds $\mu_3$ and $\mu_4$, respectively, move towards each other. Comparing their speeds, one verifies that they do not interact for times less than $T = 2T_0$.

    \begin{figure}[hbt]
		\begin{center}
			\includegraphics[trim =0mm 0mm 0mm 0mm ,scale=0.6]{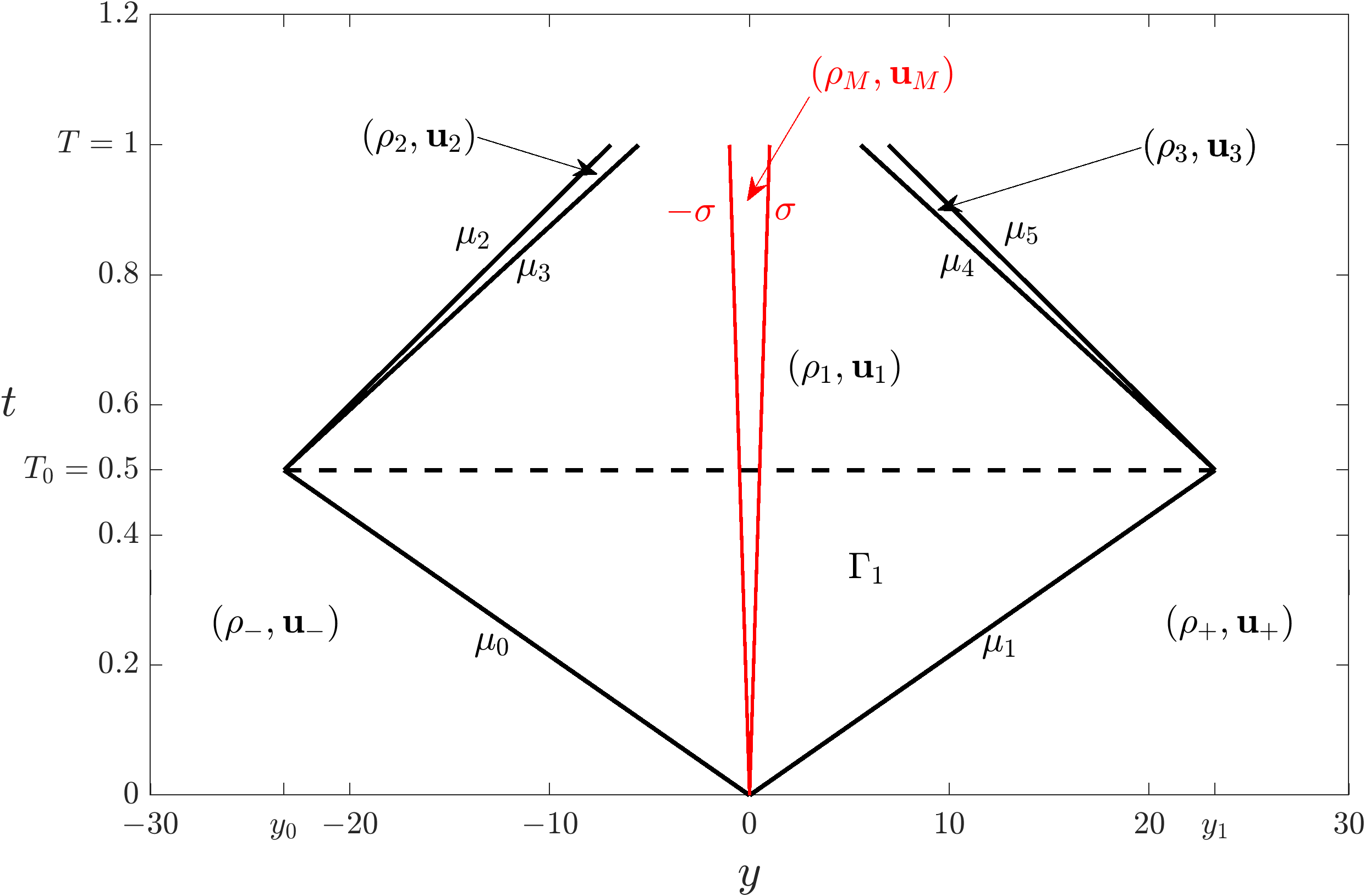}
		\end{center}
		\caption{The solution $(\rho_{\example},\uu_{\example})$ (black), which is constructed in Section~\ref{subsec:constructionCounterEx}, and the 1-D solution $(\rho_{1d},\uu_{1d})$ (red) in the $y-t$ plane for $T_0=\half$.}
		\label{fig:solutionInSpaceTime}
	\end{figure}

	\subsection{Comparison of action} \label{subsec:ActionComp} 
    
	We can now compare the action of the 1-D solution $(\rho_{1d},\uu_{1d})$ with the action of the solution $(\rho_{\example},\uu_{\example})$ constructed in Section \ref{subsec:constructionCounterEx}. To this end, we have to fix $L_1,L_2>0$, see Remark~\ref{rem:choice-of-L}. In particular, we choose $L_2>0$ such that the two solutions $(\rho_{1d},\uu_{1d})$ and $(\rho_{\example},\uu_{\example})$ coincide for all $(t,\xx)$ with $|y|>L_2$. This is achieved by setting $L_2 =-y_0$, cf.~Figure \ref{fig:solutionInSpaceTime}. Furthermore, since $(\rho_{1d},\uu_{1d})$ is independent of $x$ and $(\rho_{\example},\uu_{\example})$ can be regarded as being periodic in $x$ with period\footnote{In fact, choosing the period to be $2$ is arbitrary. Any other period is possible just as well.} $2$, we fix $L_1=1$. 
	\begin{lemma}\label{lemma:ActionComp}
        Let $T>0$ and $(\rho_{\example},\uu_{\example})$ be the solution constructed in Section \ref{subsec:constructionCounterEx} with 
        $T_0=\frac{T}{2}$. Then one has 
		\begin{align*}
			\mathcal{A}[\rho_{\example},\uu_{\example}] = K_{\example}T^2,\quad\mathcal{A}[\rho_{1d},\uu_{1d}] &= K_{1d}T^2, 
		\end{align*}
		with 
		\begin{align*}
			K_{\example} &=  \frac{25590093\sqrt{35} + 4675573\sqrt{915}}{800},\\
			K_{1d} &= \frac{3349377\sqrt{35}}{100} + \frac{6149393\sqrt{915}}{900} -\sigma \left(2\rho_M^2+\frac{1121\sqrt{1281}}{10} + \frac{28045}{6}\right).
		\end{align*}
		\begin{proof}
            One has
            \begin{align*}  
                \mathcal{A}[\rho_{\example},\uu_{\example}] = \mathcal{A}_{t<T_0} + \mathcal{A}_{t>T_0}
            \end{align*}
            with
            \begin{align*}
		          \mathcal{A}_{t<T_0}&:= \int_{0}^{T_0}\int_{-L_1}^{L_1}\int_{-L_2}^{L_2} \Big(\half\rho_\example|\uu_\example|^2 - P(\rho_\example) \Big)\dycomp\dxcomp\dt, \\
                \mathcal{A}_{t>T_0} &:= \int_{T_0}^{T}\int_{-L_1}^{L_1}\int_{-L_2}^{L_2} \Big(\half\rho_\example|\uu_\example|^2 - P(\rho_\example) \Big)\dycomp\dxcomp\dt.   
	        \end{align*}
            The ``action density'' of the constant states is given by
	        \begin{align*}
		        a_\pm := \half\rho_\pm|\uu_\pm|^2 - P(\rho_\pm),\qquad a_{1,2,3} := \half\rho_{1,2,3}|\uu_{1,2,3}|^2 - P(\rho_{1,2,3}),
	        \end{align*} 
            where $a_+=a_-$ and $a_3=a_2$. According to Proposition~\ref{prop:convexIntegration}, we have $\rho_{\example}=\rho_1$ and $|\uu_{\example}|^2 =C_1$ almost everywhere in $\Gamma_1$.
	        Hence, the ``action density'' inside $\Gamma_1$ reads
	        \begin{align*}
		          a_{\wild}:=\half\rho_{1}C_1 - P(\rho_1). 
	        \end{align*}
	        We compute
            \begin{align}
            \begin{split}
	       	    \mathcal{A}_{t<T_0}&=4\int_0^{T_0}\int_{-L_2}^{0} \Big(\half\rho_{\example}|       \uu_{\example}|^2 - P(\rho_{\example})\Big)\; \dycomp \dt\\ 
		          &= 4\int_0^{T_0}\left(\int_{y_0}^{\mu_0t} a_{-}\; \dycomp + \int_{\mu_0t}^{0} a_{\wild}\; \dycomp  \right)\dt\\
		          %&= 4L_1\int_0^{T_0}\left((\mu_0t - y_0) A_- - \mu_0 A_{\wild}t \right)\dt\\
		          %&= 4L_1\int_0^{T_0}\left(\mu_0(A_- - A_{\wild})t -y_0A_-\right)\dt\\
                %&= 4L_1\left(\half\mu_0(A_- - A_{\wild})t^2  -y_0A_-t\right)\bigg |_{t=0}^{t=T_0}\\
                &= -\half \mu_0\left(a_- + a_{\wild}\right)T^2.
            \end{split}\label{eq:calculationIntegral1}
	        \end{align}	
            Furthermore, one has
	        \begin{align}
            \begin{split}
		          \mathcal{A}_{t>T_0}&=4\int_{T_0}^T\int_{-L_2}^{0} \Big(\half\rho_{\example}|\uu_{\example}|^2 - P(\rho_{\example})\Big)\; \dycomp \dt\\\
		          &= 4\int_{T_0}^T\left(\int_{y_0}^{\mu_2(t-T_0)+y_0} a_-\; \dycomp + \int_{\mu_2(t-T_0)+y_0}^{\mu_3(t-T_0)+y_0} a_2\; \dycomp  + \int_{\mu_3(t-T_0)+y_0}^{0} a_1\; \dycomp \right)\dt\\
		         % &= 4L_1\int_{T_0}^T\left( \mu_2(t-T_0)A_- + (\mu_3-\mu_2)(t-T_0)A_2 - (\mu_3(t-T_0) +y_0)A_1  \right) \dt\\
               % &= 4L_1\int_{T_0}^T  (\mu_2A_- + (\mu_3-\mu_2)A_2 - \mu_3A_1 )t +  (-\mu_2T_0A_- - (\mu_3-\mu_2)T_0A_2 + (\mu_3T_0-y_0)A_1) \dt \\     
		         % &= 4L_1\left( \half(\mu_2A_- + (\mu_3-\mu_2)A_2 - \mu_3A_1 )t^2 +  (-\mu_2T_0A_- - (\mu_3-\mu_2)T_0A_2 + (\mu_3T_0-y_0)A_1)t  \right)\bigg |_{t=T_0}^{t=T}\\
                &=-\half ( 2\mu_0a_1 + \mu_2(a_2-a_-)+\mu_3(a_1-a_2) )T^2.
            \end{split}\label{eq:calculationIntegral2}
	        \end{align}  
	        All coefficients involved in \eqref{eq:calculationIntegral1}, \eqref{eq:calculationIntegral2} are given explicitly by the construction in Section \ref{subsec:constructionCounterEx}. Simple computations yield 
		    \begin{align*}
		       	a_- &= \frac{1121\sqrt{1281}}{20} + \frac{28045}{12}, & a_{\wild} &= \frac{3363\sqrt{1281}}{40} + \frac{27965}{8},\\
		        a_1&=-9, & a_2&=\frac{61029}{2},
		      \end{align*}
	        from which one can simply verify that $\mathcal{A}[\rho_{\example},\uu_{\example}] = K_{\example}T^2$ for $K_{\example}$ as given in the assertion. 
            
            Similar to the above computations we set
	        \[
	        a_M:=\half \rho_M|\uu_M|^2 - P(\rho_M),
	        \]
	        and calculate
	        \begin{align*} 
		          \mathcal{A}[\rho_{1d},\uu_{1d}]&=4\int_0^T\int_{-L_2}^{0} \Big(\half\rho_{1d}|\uu_{1d}|^2 - P(\rho_{1d}) \Big)\dycomp \dt\\\
		          &= 4\int_0^T\left(\int_{y_0}^{-\sigma t} a_-\; \dycomp + \int^{0}_{-\sigma t} a_M\; \dycomp \right)\dt \\
		       % &= 4L_1\int_0^T\left( -(\sigma t+y_0)A_-  + \sigma tA_M \right) \dt\\
              % &= 4L_1\int_0^T\left( \sigma(A_M-A_-)t -A_-y_0\right) \dt\\
		         % &= 4L_1\left(\half\sigma(A_M-A_-)t^2 - A_-y_0t \right)\bigg |_{t=0}^{t=T}\\
                 &=  2\left( \sigma(a_M - a_-) - \mu_0a_-\right)T^2.
	        \end{align*}	
            Inserting the explicit value for $a_-$ and $\mu_0$ as well as $a_M=-\rho_M^2$ yields the assertion. 
	       \end{proof}
    \end{lemma}

    We are now finally able to prove the main theorem of this paper.
	\begin{proof}[Proof of Theorem~\ref{thm:mainTheorem}]
        Let $T>0$ and $(\rho_{\example},\uu_{\example})$ be the solution constructed in Section \ref{subsec:constructionCounterEx} with
        $T_0=\frac{T}{2}$.
		Lemma \ref{lemma:ActionComp} yields that 
		\begin{align*}
			\mathcal{A}[\rho_{\example},\uu_{\example}]< \mathcal{A}[\rho_{1d},\uu_{1d}] 
		\end{align*}
		if and only if
		\begin{align*} 
		      K_{\example}-K_{1d} = \left(2\rho_M^2 + \frac{1121\sqrt{1281}}{10}+\frac{28045}{6} \right)\sigma - \frac{1204923\sqrt{35}}{800} -\frac{7114987\sqrt{915}}{7200}<0,
		\end{align*}
        which can easily be verified using \eqref{eq:estimateRhoMSigma}. So, $(\rho_{\example},\uu_{\example})$ indeed beats the 1-D solution $(\rho_{1d},\uu_{1d})$ in the sense that it has a smaller action. Consequently, the 1-D solution does not satisfy the least action admissibility principle.
	\end{proof}

\section{Concluding remarks} \label{sec:concluding-remarks}
For each fixed $T>0$ used to define the action \eqref{eq:actionDef}, we have employed convex integration to construct an admissible weak solution $(\rho_{\example},\uu_{\example})$ that beats the 1-D solution $(\rho_{1d},\uu_{1d})$ in terms of the least action admissibility principle. The decisive mechanism for the solution $(\rho_{\example},\uu_{\example})$ to have a lower action than $(\rho_{1d},\uu_{1d})$ is the downward jump in kinetic energy along a surface in space time, see Remark~\ref{rem:glueingSolutions}. To illustrate our strategy to find $(\rho_{\example},\uu_{\example})$, we consider the function
\[
	A[\rho,\uu](t):=\int_{-L_1}^{L_1}\int_{-L_2}^{L_2} \Big(\half\rho|\uu|^2-P(\rho)\Big) \dycomp \dxcomp,
\]
for any admissible weak solution $(\rho,\uu)$, so that 
$$
    \mathcal{A}[\rho,\uu]= \int_0^T A[\rho,\uu](t) \dt.
$$
As can be easily seen from the proof of Lemma~\ref{lemma:ActionComp}, the function $t\mapsto A[\rho_{\example},\uu_{\example}](t)$ is piecewise linear, see Figure \ref{fig:2a}. As the kinetic energy undergoes a jump at time $T_0$ from a higher to a lower value, cf.~Remark \ref{rem:glueingSolutions}, the function $t\mapsto A[\rho_{\example},\uu_{\example}](t)$ exhibits the same behavior. 

\begin{figure}[htb] 
	\centering
	\subfloat[{The functions $A[\rho,\uu]$.}\label{fig:2a}]{
		\centering
		\includegraphics[width=0.43\textwidth]{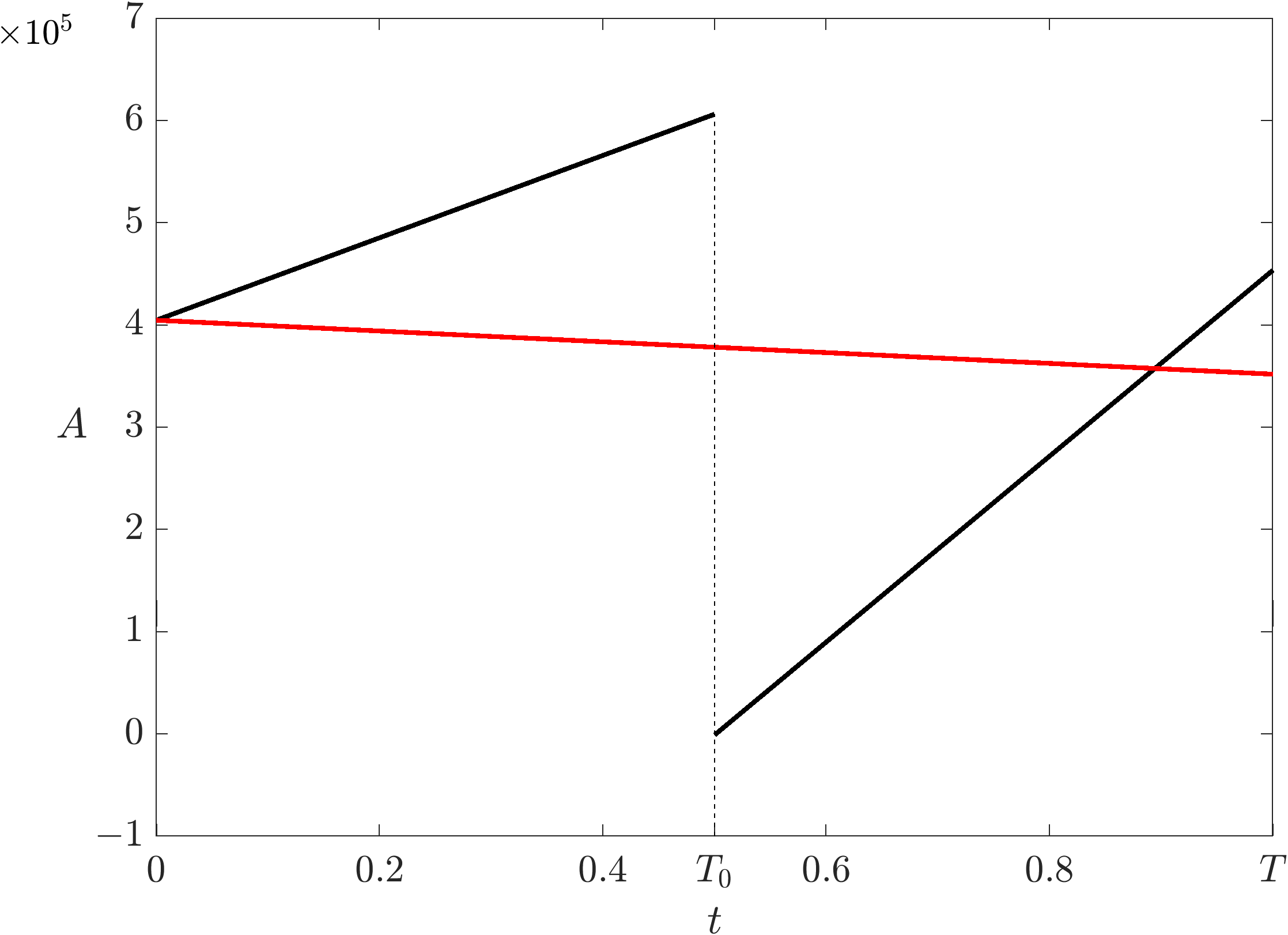}
	}
	\hspace{1.5cm}
	\subfloat[{The functions $\widetilde{\mathcal{A}}[\rho,\uu]$.}\label{fig:2b}]{
		\centering
		\includegraphics[width=0.43\textwidth]{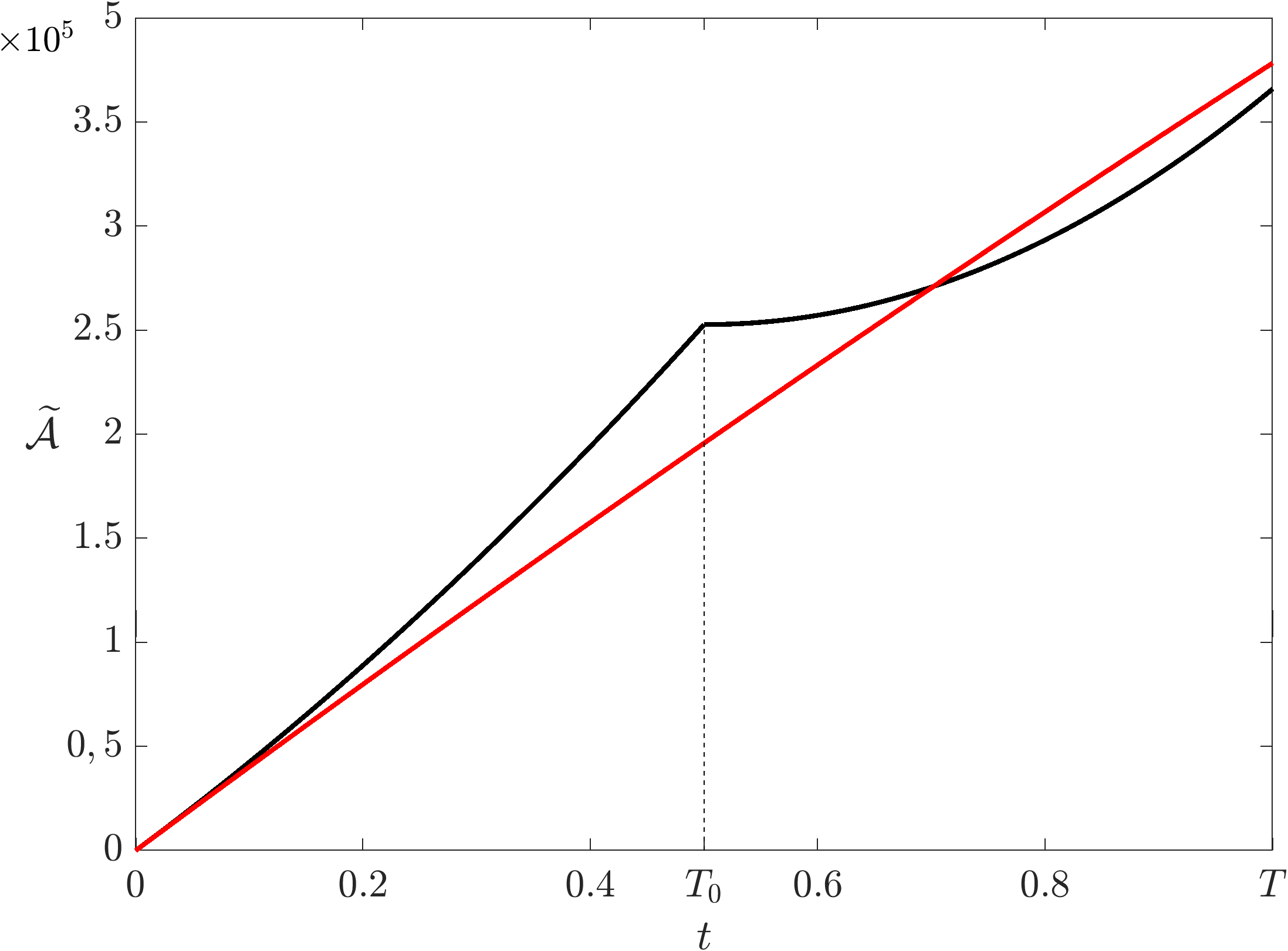}
	}
	\caption{The functions $A[\rho,\uu]$ and $\widetilde{\mathcal{A}}[\rho,\uu]$ for the solutions $(\rho_{1d},\uu_{1d})$ (red) and $(\rho_{\example},\uu_{\example})$ (black), respectively, where $T=1$.} 
	\label{fig:RiemannSolutionPhaseSpace}
\end{figure} 

Let us next consider the action on the time interval $[0,t]$ as a function dependent on the final time,
\begin{align*}
	\widetilde{\mathcal{A}}[\rho,\uu](t) := \int_{0}^{t}A[\rho,\uu](s)\ds, 
\end{align*}
so that $\mathcal{A}[\rho,\uu]= \widetilde{\mathcal{A}}[\rho,\uu](T)$. Since $t\mapsto A[\rho_{\example},\uu_{\example}](t)$ is piecewise linear, the function $t\mapsto \widetilde{\mathcal{A}}[\rho_{\example},\uu_{\example}](t)$ is piecewise quadratic as shown in Figure \ref{fig:2b}. For $t<T_0$, the solutions constructed in Section~\ref{sec:wildSolutions} coincide with the solutions generated in \cite{ChiKre14}. Those solutions are known to have a larger action than the 1-D solution, cf.~\cite{GGKS24pre}. Consequently, 
\begin{equation} \label{eq:dep-finaltime1}
    \widetilde{\mathcal{A}}[\rho_{\example},\uu_{\example}](t) > \widetilde{\mathcal{A}}[\rho_{1d},\uu_{1d}](t)\qquad \text{ for }t<T_0,
\end{equation}
i.e.~the action of the convex integration solution is larger than the action of the 1-D solution, see Figure~\ref{fig:2b}. However, as time increases, the convex integration solution beats the 1-D solution, i.e.
\begin{equation} \label{eq:dep-finaltime2}
    \widetilde{\mathcal{A}}[\rho_{\example},\uu_{\example}](t) < \widetilde{\mathcal{A}}[\rho_{1d},\uu_{1d}](t)\qquad \text{ for }T_0\ll t\leq T,
\end{equation}
see again Figure~\ref{fig:2b}. In fact, it is the jump in kinetic energy (or, equivalently, in $A[\rho_{\example},\uu_{\example}]$) at time $T_0$ that yields a kink in the function $t\mapsto \widetilde{\mathcal{A}}[\rho_{\example},\uu_{\example}](t)$ which finally leads to a lower action as desired. 

Our example $(\rho_{\example},\uu_{\example})$ also shows that the question which solution is preferred by the least action admissibility principle depends on the final time, see \eqref{eq:dep-finaltime1} and \eqref{eq:dep-finaltime2} above. This is another counterintuitive fact which suggests that the least action admissibility principle is not the right selection principle in the context of the Euler equations \eqref{eq:euler}.

\section*{Acknowledgements}
Funded by the Deutsche Forschungsgemeinschaft (DFG, German Research Foundation) - SPP 2410 \emph{Hyperbolic Balance Laws in Fluid Mechanics: Complexity, Scales, Randomness (CoScaRa)}, within the Project 525935467 \emph{Convex integration: towards a mathematical understanding of turbulence, Onsager conjectures and admissibility criteria}.

\printbibliography[heading=bibintoc]

\end{document}